\documentclass[10pt]{amsart}

\usepackage{amsfonts}
\usepackage{amscd}
\usepackage{amsmath}
\usepackage{amsthm}
\usepackage{amssymb}
\usepackage{amsxtra}

\input xy
\xyoption{all}

\newcommand{\bL}{\mathbf{L}}
\newcommand{\bR}{\mathbf{R}}
\newcommand{\bbT}{\mathbb{T}}

\newcommand{\caC}{{\mathcal C}}

\newcommand{\caA}{{\mathcal A}}

\newcommand{\caS}{{\mathcal S}}

\newcommand{\caT}{{\mathcal T}}

\newcommand{\caL}{{\mathcal L}}

\newcommand{\caK}{{\mathcal K}}

\newcommand{\integers}{\mathbf{Z}}

\newcommand{\bH}{\mathbb{H}}
\newcommand{\bQ}{\mathbb{Q}}
\newcommand{\bA}{\mathbb{A}}

\newcommand{\Hom}{\mathrm{Hom}}
\newcommand{\Ho}{\mathsf{Ho}}

\newcommand{\colim}{\mathrm{colim}}

\newcommand{\Comm}{\mathrm{Comm}}

\newcommand{\LQ}{\mathsf{L}\mathbf{Q}}

\newcommand{\shvnis}{\mathsf{Shv}_{\mathit{Nis}}}

\newcommand{\smcor}{\mathsf{SmCor}}

\newcommand{\Cpx}{\mathsf{Cpx}}
\newcommand{\ztr}{\integers_{\mathit{tr}}}
\newcommand{\zequi}{\mathrm{z}_{\mathit{equi}}}

\theoremstyle{theoremstyle}
\newtheorem{theorem}{Theorem}[section]
\newtheorem*{theorem*}{Theorem}
\newtheorem{lemma}[theorem]{Lemma}
\newtheorem{proposition}[theorem]{Proposition}
\newtheorem*{proposition*}{Proposition}
\newtheorem{corollary}[theorem]{Corollary}
\newtheorem*{corollary*}{Corollary}

\newtheorem*{conjecture*}{Conjecture}

\newtheorem{definition}[theorem]{Definition}
\newtheorem{definition*}{Definition}
\newtheorem{remark}[theorem]{Remark}
\newtheorem{remark*}{Remark}

\newcommand{\D}{\mathsf{D}}

\newcommand{\PP}{\mathbf{P}}

\newcommand{\Gr}{\mathbf{Gr}}

\newcommand{\KGL}{\mathsf{KGL}}
\newcommand{\MGL}{\mathsf{MGL}}
\newcommand{\PMGL}{\mathsf{PMGL}}

\newcommand{\MU}{\mathsf{MU}}

\newcommand{\MZ}{\mathsf{M}\mathbf{Z}}

\newcommand{\Sm}{\mathsf{Sm}}

\newcommand{\unit}{\mathbf{1}}

\newcommand{\Thom}{\mathsf{Th}}

\newcommand{\mMod}{\mathrm{-Mod}}
\newcommand{\id}{\mathrm{id}}

\renewcommand{\S}{\mathbb{S}}

\newcommand{\sP}{\mathsf{P}}
\newcommand{\DM}{\mathsf{DM}}

\title{Periodizable motivic ring spectra}
\author{Markus Spitzweck}
\date{\today}
\begin{document}

\pagestyle{plain}
\maketitle

\begin{abstract}
We show that the cellular objects in the module category over a motivic $E_\infty$-ring
spectrum $E$ can be described as the module category over a graded topological spectrum
if $E$ is strongly periodizable in our language. A similar statement is proven
for triangulated categories of motives. Since $\MGL$ is strongly periodizable we
obtain topological incarnations of motivic Landweber spectra. Under some categorical
assumptions the unit object of the model category
for triangulated motives is as well strongly periodizable giving motivic cochains whose
module category models integral triangulated categories of Tate motives.
\end{abstract}

\tableofcontents

\section{Introduction}

In \cite{km} graded $E_\infty$-algebras have been constructed whose module
categories are candidates for triangulated categories of Tate motives over a given field.
Since then many approaches to triangulated categories of motives were developed,
most notably Voevodsky's approach \cite{voevodsky.triangulated}.
Thus the question arises if one can directly construct graded $E_\infty$-algebras
from these motivic categories modelling Tate motives. Among other things we give
a solution to this problem, modulo standard categorical assumptions.
In \cite{joshua} $E_\infty$-motivic cochains have been
constructed (but see \cite{may.sheaf-coh}) without adressing the comparison of the module category to Tate motives.

In \cite{spitzweck-nistech} and \cite{spitzweck-mot} rational cycle complexes have been constructed
whose module categories model rational triangulated categories of Tate motives, for a summary
see \cite[II.5.5.4, Th. 111, II.5.5.5]{levine.survey-mixed-motives}. In this paper we give generalizations of
these constructions to integral triangulated categories of Tate motives.

Our approach emphasizes the notion of a strongly periodizable $E_\infty$-algebra.
In the motivic context a strong periodization of an $E_\infty$-ring spectrum $E$
is a graded $E_\infty$-ring spectrum $P$ such that in the stable homotopy category we have
an isomorphism
\begin{equation} \label{per-look}
P \cong \bigvee_{i \in \integers} \Sigma^{2i,i} E
\end{equation}
with the obvious multiplication. Here $\Sigma^{p,q}$ is the usual motivic shift functor of simplicial
degree $p$ and Tate degree $q$. Note that in order the right hand side of (\ref{per-look})
to be a commutative monoid in the motivic stable homotopy category we need some assumptions
on $E$ or the base scheme, namely the map $E \wedge T^2 \to E \wedge T^2$ which is the twist in
the second variable should be the identity.

Theorem (\ref{rep-th}) implies that if $E$ admits a strong periodization then the cellular objects
in the derived category of $E$-modules have a description in terms of a module category
over a graded topological $E_\infty$-ring spectrum. A similar statement is true for $E$
an $E_\infty$-ring object in the category of motives over a given base.

Thus to find good representations of cellular objects it is necessary to prove
strong periodizability of a given $E_\infty$-ring object. In section \ref{examples}
we do this for the motivic cobordism spectrum $\MGL$ and the unit object in the category
of motives over a field $k$ of characteristic $0$. The construction of the former
is a generalization of the construction of a strict commutative ring model of $\MGL$ in \cite{PPR2}.
The strategy for the latter is as follows:
first we construct a semi periodization, i.e. an $E_\infty$-algebra $P$
such that $$P \cong \bigoplus_{i \le 0} \integers(i)[2i]$$
as algebra in the triangulated category of motives over $k$ using explicit cycle groups.
Then we employ a localization technique (proposition (\ref{semi-per})) to construct a strong
periodization. It is here where we need some assumptions from the theory
of $\infty$-categories. We summrize these in section \ref{prelim}.
Under these assumptions we thus prove that there is a graded $E_\infty$-algebra
in complexes of abelian groups whose derived category of modules is equivalent
as tensor triangulated category to the
full subcategory of Tate motives in Voevodsky's category of big motives,
see corollary (\ref{mot-rep}).

Likewise we obtain a representation theorem for the full subcategory
of cellular objects in the derived category of $\MGL$-modules, see corollary
(\ref{mgl-rep}). As another corollary, under our categorical assumptions, we obtain the strong periodizability of
the motivic Eilenberg MacLane spectrum over perfect fields due to the work
of Voevodsky \cite{voevodsky-zero-slice} and Levine \cite{levine-htp} on
the zero slice of the sphere spectrum, see corollary (\ref{mz-rep}).

Since motivic Landweber spectra have incarnations as cellular highly structured $\MGL$-modules
\cite{NSO1} we thus obtain topological models of these motivic Landweber spectra.

Here is an overview of the sections. In section \ref{per-alg} we first give general background on
$E_\infty$-algebras in model categories. Then we give the definition of being
strongly periodizable and prove the abstract representation theorem (\ref{rep-th}). Moreover
we show that under our categorical assumptions the existence of a semi periodization implies
the existence of a periodization, proposition (\ref{semi-per}). Finally we show that every algebra
which receives a map from a strongly periodizable algebra is itself strongly periodizable,
proposition (\ref{alg-periodizable}). Section \ref{symm-bimor} contains
the technical part to show that our constructions for $\MGL$ and the unit sphere in motives
are indeed strong periodizations. Section \ref{examples} contains our examples
of strongly periodizable algebras and the applications to representation theorems.

{\bf Acknowledgements.} The author thanks Christian Blohmann, Ulrich Bunke, Oliver R\"ondigs and Ansgar Schneider
for useful discussions on the subject.

\section{Preliminaries} \label{prelim}

In this text we have to deal with commutative algebras in a homotopical
setting. In special cases one can directly work with commutative algebras,
e.g. in symmetric spectra with the positive model structure. We have chosen
to use the language of $\S$-modules as in \cite{ekmm} or \cite{km}
and the general setting of \cite{spitzweck-thesis} which is
adapted to an abstract formulation of the problem we are discussing.

We will freely deal with the language of $\infty$-categories as
introduced in \cite{lurie-topoi}. We will make the assumption that
the $\infty$-categories associated to the model or semi model categories
appearing in this text are presentable in the language of \cite{lurie-topoi}.
Among other things this enables us to localize these $\infty$-categories, see \cite{lurie-topoi}.
Note also that the theory of presentable $\infty$-categories is equivalent
to the theory of combinatorial model categories \cite[Rem. 5.5.1.5]{lurie-topoi}.

Also we will assume that the theory of algebras
in the (semi) model category setting and the $\infty$-category setting
are compatible. E.g. we will assume that the $\infty$-category associated
to the semi model category of $E_\infty$-algebras in a given suitable symmetric
monoidal model category $\caC$ is equivalent to the $\infty$-category of commutative algebras
in the symmetric monoidal $\infty$-category associated to $\caC$.
The same applies in the relative setting of algebras over a given algebra.

We also suppose that the $\infty$-categories of modules which appear are finitely
generated and stable. In particular the associated triangulated categories will
be compactly generated.

We call these assumptions our {\em categorical assumptions}.

We use them only twice when we localize a semi periodization of
a given $E_\infty$-algebra to obtain a periodization and when we talk about
the zero slice of $\MGL$ as a motivic $E_\infty$-ring spectrum.

\section{Conventions}

If $\caC$ is a category and $A \in \caC^\integers$ a graded object we will write $A_r$ for the object
in degree $r$. We let $A(r)$ be the shift given by $A(r)_k=A_{k-r}$.

An {\em $\Omega$-spectrum} will be a spectrum $X$ such that the derived adjoints
of the structure maps, $X_n \to \underline{\bR \Hom}(K,X_{n+1})$, are equivalences.
Here $K$ is the object with which we build the spectra.

When dealing with symmetric spectra we have to be very careful about the symmetric
group actions, for the convenience of the reader we refer for that to the Manipulation
rules for coordinates, \cite[Remark I.1.12]{schwede.book}.

\section{Periodizable $E_\infty$-algebras} \label{per-alg}

In this section we develop the abstract context in which categories
of cellular objects will be modelled by modules over graded $E_\infty$-ring spectra resp. algebras.

Let $\caC$ be a cofibrantly generated left proper symmetric monoidal model category.
We assume that the domains of the generating sets $I$ and $J$ are small relative
to the whole category and that the tensor unit and the domains of the maps in $I$ are cofibrant.

Let $\caS$ be the category of symmetric spectra in simplicial sets equipped with the stable
projective model structure. Also let $\caA$ be the category of (unbounded) chain complexes of
abelian groups equipped with the projective model structure. Both categories fulfill the assumptions
for $\caC$.

In the whole section we will assume that $\caC$ either receives a symmetric monoidal left Quillen functor $l$
from $\caS$ or from $\caA$. We denote by $\caL$ the image of the linear isometries
operad either in $\caS$, $\caA$ or $\caC$, depending in which category we talk about $\caL$-algebras
or $\S$-modules.

We set $\S:=\caL(1)$ which is a monoid. W let $\S \caS$, $\S \caA$ and $\S \caC$ be the categories
of $\S$-modules in the respective categories. By \cite[Proposition 9.3]{spitzweck-thesis}
these are symmetric monoidal model categories with weak unit. The tensor product is given by
$M \boxtimes N = \caL(2) \otimes_{\S \otimes \S} (M \otimes N)$. The pseudo tensor unit is $\unit_\caS$, $\unit_\caA$,
$\unit_\caC$ resp. equipped with the trivial $\S$-module structure.

For the discussion of commutative algebras we only treat the case of $\caC$ since those for $\caS$ and $\caA$
are special cases thereof. We write $\Comm(\caC)$ for the category of $\caL$-algebras in $\caC$. This is the same
as the category of commutative monoid objects in the symmetric monoidal category of unital
$\S$-modules, see \cite[Proposition 9.4]{spitzweck-thesis}. $\Comm(\caC)$ is a cofibrantly
generated semi model category by \cite[Corollary 9.7]{spitzweck-thesis}.

For $A \in \Comm(\caC)$ we write $\Comm(A)$ for algebras under $A$. By loc. cit. it is a semi model category
for cofibrant $A$. For a map $f \colon A \to B$ between cofibrant algebras the induced map $\Comm(A) \to \Comm(B)$
is a left Quillen functor which is an equivalence if $f$ is.

We denote by $A \mMod$ the category of $A$-modules. It is a symmetric monoidal category with pseudo unit,
see \cite[after Def. 9.8]{spitzweck-thesis}. If $A$ is cofibrant then by \cite[Proposition 9.10]{spitzweck-thesis}
it is a symmetric monoidal model category with weak unit. By loc. cit. if $f \colon A \to B$ is a map between
cofibrant algebras then the push forward $f_*$ is a symmetric monoidal left Quillen functor which is a Quillen
equivalence if $f$ is an equivalence.

We set $\D(A):=\Ho (QA \mMod)$ where $QA \to A$ is a cofibrant replacement.

We next introduce graded objects. Again we treat the case for $\caC$.

We let $\caC^\integers$ be the category of $\integers$-graded ojects in $\caC$, i.e.
the $\integers$-fold product of $\caC$ with itself. We employ the symmetric monoidal structure
on $\caC^\integers$ which is given on objects by
$((a_i)_{i \in \integers} \otimes ((b_j)_{j \in \integers})_k= \bigsqcup_{i+j=k} a_i \otimes b_j$.

With these definitions $\caC^\integers$ satisfies the same assumptions as $\caC$, in particular it is a
cofibrantly generated symmetric monoidal model category. Thus the above discussion for $\caC$
applies likewise to $\caC^\integers$.

From now on we will fix a cofibrant object $K \in \caC$ which is $\otimes$-invertible in $\Ho \caC$.

Let $A$ be a commutative monoid in $\Ho \caC$. We define its periodization $P(A)$ to be the monoid
in $\Ho \caC$ with underlying object $\bigsqcup_{i \in \integers} A \otimes K^{\otimes i}$ and multiplication
induced by the maps $(A \otimes K^{\otimes i}) \otimes (A \otimes K^{\otimes j}) \cong A \otimes A \otimes K^{\otimes (i+j)}
\to A \otimes K^{\otimes (i+j)}$.
When we consider this periodization we will always assume
that $\id_A \otimes \tau \colon A \otimes K^{\otimes 2} \to A \otimes K^{\otimes 2}$, $\tau \colon K^{\otimes 2} \to K^{\otimes 2}$ the twist,
is the identity.
Thus $P(A)$ becomes a commutative monoid.
The periodization $P(A)$ can be viewed as a commutative monoid in
$(\Ho \caC)^\integers$. Let $E \in \Comm(\caC)$. Then we can construct the periodization
as monoid in $\D(E)^\integers$. We denote this also by $P(E)$.

Next observe that there is a symmetric monoidal left Quillen functor $i \colon \caC \to \caC^\integers$ sending
$X$ to the sequence $(\ldots,\emptyset,\emptyset,X,\emptyset,\emptyset , \ldots)$, where $X$ sits in degree $0$.

\begin{definition}
Let $E \in \Comm(\caC)$. The algebra $E$ is called {\em strongly periodizable} if there is an algebra $P \in
\Ho(\Comm(Qi(E)))$, $Qi(E) \to i(E)$ a cofibrant replacement,
such that $P$ becomes isomorphic to $P(E)$ as monoid in $\D(E)^\integers$ under the image of $i(E)$.
A map $i(E) \to P$ satisfying this assumption will be called a {\em strong periodization} of $E$.
\end{definition}

Note that if $f \colon i(E) \to P$ is a strong periodization then $f$ induces an equivalence in degree $0$.

For a commutative monoid $A$ in $\Ho \caC$ we let $P_+(A)$ be the algebra
$\bigsqcup_{i \in \integers_{\ge 0}} A \otimes K^{\otimes i}$ in $(\Ho \caC)^{\integers_{\ge 0}}$ and
$P_-(A)$ be the algebra
$\bigsqcup_{i \in \integers_{\le 0}} A \otimes K^{\otimes i}$ in $(\Ho \caC)^{\integers_{\le 0}}$ (here again we make implicitely
the assumption on the twist).

For $E \in \Comm(\caC)$ we denote by $P_\pm(E)$ also the corresponding algebras in $\D(E)^{\integers_{\gtreqless 0}}$.

We let $i_\pm \colon \caC \to \caC^{\integers_{\gtreqless 0}}$ be the canonical
symmetric monoidal left Quillen functors.

We say that an $E \in \Comm(\caC)$ admits a {\em semi periodization} if there is an algebra $P \in \Ho(\Comm(Qi_\pm(E)))$
such that $P$ becomes isomorphic to $P_\pm(E)$ as monoid in $\D(E)^{\integers_{\gtreqless 0}}$ under the image of $i_\pm(E)$.

\begin{proposition} \label{semi-per}
 Suppose that $E \in \Comm(\caC)$ admits a semi periodization. Suppose our categorical assumptions hold.
Then $E$ is strongly periodizable.
\end{proposition}

\begin{proof}
 This is the standard technique of inverting elements in $E_\infty$-ring objects.
Wlog we handle the case that $E$ has a semi periodization in $\Comm(\caC^{\integers_{\ge 0}})$.
Let $Qi_+(E) \to P$ be such a semi periodization, $Qi_+(E)$ a cofibrant
replacement. We denote by $P$ also the image of $P$
with respect to the functor $\Comm(\caC^{\integers_{\ge 0}}) \to \Comm(\caC^\integers)$.
Thus we have a map $Qi(E) \to P$ in $\Comm(\caC^\integers)$.
The element in $P$ we want to invert is the map $a \colon i(K)(1) \to P$ in $(\Ho \caC)^{\integers}$
corresponding to the map $K \to A \otimes K$ given by the unit of $A$ and the identity on $K$.

We let $\kappa \colon i(K)(1) \otimes P \to P$ be the map in $\D(P)$ given by multiplication with $a$.
For a lift $\tilde{\kappa}$ of $\kappa$ to $QP \mMod$ for a cofibrant replacement $QP \to P$
as a map between cofibrant $QP$-modules
we consider the free $QP$-algebra map $F_{QP}(\tilde{\kappa})$ on $\tilde{\kappa}$.
We denote by $L_\mathrm{alg}$ the localization functor on the $\infty$-category associated
to $\Comm(QP)$ which inverts $F_{QP}(\tilde{\kappa})$. We claim that a local model
of $QP$ with respect to $L_\mathrm{alg}$ will yield a strong periodization of $E$.

Therefore we first localize the category of $P$-modules. We let $L$ denote
the localization functor on the $\infty$-category associated to $QP \mMod$ which inverts
$\kappa$. On homotopy categories it has the same effect as localizing
$\D(P)$ with respect to the full localizing triangulated subcategory spanned by the (co)fiber of
$\kappa$. For any $M \in \D(P)$ we denote $\kappa_M$ the map $\kappa \otimes_P \id_M$
or suitable twists by tensor powers of $i(K)(1)$ thereof. The local objects in $\D(P)$ are
exactly the modules $M$ such that $\kappa_M$ is an isomorphism. Moreover, by adjunction, the
local objects in $\Ho(\Comm(QP))$ are the algebras which are local as $P$-modules.

Since the $\infty$-category of $P$-modules is finitely generated it is easily seen that
$L$ is given by a (homotopy) colimit
\begin{equation} \label{loc-colim}
M \mapsto LM=\colim(M \overset{\kappa_M}{\to} M \otimes i(K)^{-1}(-1)
\overset{\kappa_M}{\to} M \otimes i(K)^{-2}(-2) \to \cdots)
\end{equation}

(the transition maps are local equivalences, thus the map from $M$ to the colimit
is also a local equivalence, and the colimit is local by a finite generation argument).

Thus we can write $LM=M \otimes_P LP$.
It follows that $L$ is compatible with the tensor product
on the triangulated categories and in the $\infty$-categorical setting
(for the latter see \cite[Def. 1.28, Prop. 1.31]{lurie-symm} and the discussion in
\cite[par. 5]{gepner-snaith}).

We let $\caK \to \caL$ be the symmetric monoidal localization functor from the $\infty$-category
of $P$-modules to the local objects. We get an induced adjunction $$F \colon \Comm(\caK) \leftrightarrow \Comm(\caL) \colon G$$
on commutative algebras in $\caK$ and $\caL$.
We claim this is the localization at the morphism $F_{QP}(\tilde{\kappa})$.
First $G$ is a full embedding: the counit is an isomorphism since it is so on the underlying modules.
We see the image of $G$ is exactly the subcategory of $F_{QP}(\tilde{\kappa})$-local objects which settles the claim.

Thus we get a $F_{QP}(\tilde{\kappa})$-local model of $P$ by the unit $P \to GF(P)$.

We have to detect the algebra structure of $GF(P)$ as algebra in $\D(E)^\integers$.

First on the level of modules we have $LP \otimes_P LP \cong LP$. This makes $LP$ into an algebra in $\D(P)$.
(This is what is called a $P$-ring spectrum in \cite{ekmm}.) This algebra $LP$ is clearly the image of $GF(P)$
in commutative monoids in $\D(P)$. By forgetting the $P$-module structure (i.e. applying the lax symmetric monoidal
functor $\D(P) \to \D(E)^\integers$) we get the algebra we want to know. 

Since the localization $L$ is given by the colimit (\ref{loc-colim}) we see that $LP$ as
module over $i_>(P_+(E))$, $i_>$ the functor $\D(E)^{\integers_{\ge 0}} \to \D(E)^\integers$, has
the form $(\ldots,E\otimes K^{-2},E \otimes K^{-1}, E , E \otimes K, E \otimes K^2, \ldots)$
with the obvious multiplication (note that the maps from the stages of the colimit to the colimit have
to be compatible with the multiplication).

Since on the level of model categories the tensor product $LP \otimes_P LP$
is gotten from $LP \otimes_{i(E)} LP$ by a coequalizer diagram we exhibit a diagram
$$LP \otimes_{i(E)} P \otimes_{i(E)} LP \rightrightarrows LP \otimes_{i(E)} LP \to LP \otimes_P LP$$
(all tensor products are derived tensor products).
On the degree $0$ part of the image of $GF(P)$ in $\D(E)^\integers$ we know already the multiplication
on the whole of the image of $GF(P)$ because of the unitality property.

The above diagram then forces the multiplication on the image of $GF(P)$ to be the one claimed.
\end{proof}

Next the left Quillen functor $l \colon \caS \to \caC$ resp. $l \colon \caA \to \caC$ comes into play.
We denote ly $l^\integers$ the prolonguation of $l$ to $\integers$-graded objects. We denote by $r$ the
right adjoint to $l$, thus $r^\integers$ is the right adjoint to $l^\integers$.

Given an $E_\infty$-algebra $E$ in $\caC^\integers$ we can look at its imgae under $r^\integers$
and study the module category of this algebra.

\begin{theorem} \label{rep-th}
 Let $E \in \Comm(\caC)$ be cofibrant and let $g \colon i(E) \to P$ be a strong periodization.
Assume wlog that $g$ is a cofibration and $P$ is fibrant. Then $\D(r^\integers(P))$ is canonically equivalent to the localizing
full triangulated subcategory of $\D(E)$ spanned by the spheres $E \otimes K^i$, $i \in \integers$,
as tensor triangulated category. Moreover this equivalence comes from Quillen functors between model
categories.
\end{theorem}

\begin{proof}
We treat the case where $l \colon \caS \to \caC$, the case
$l \colon \caA \to \caC$ is analogous.

 Let $Qr^\integers(P) \to r^\integers(P)$ be a cofibrant replacement. Let $f$ be the composition
$l^\integers(Qr^\integers(P)) \to l^\integers(r^\integers(P)) \to
P$. Let $M \in Qr^\integers(P) \mMod$ be cofibrant. Then $f_*(l^\integers(M)) \in P \mMod$.
Let further $v$ be the functor which sends a graded object $X \in \caC^\integers$ to $X_0 \in \caC$.
The functor $v$ can be made into a lax symmetric monoidal functor, i.e. there are associative, commutative
and unital transformations $v(X) \otimes v(Y) \to v(X \otimes Y)$. This also extends to $\S$-modules.
In particular $v$ sends a $P$-module to a $P_0$-module and via pullback along $g_0$ to an $E$-module.
Thus $v(f_*(l^\integers(M))) \in E \mMod$. Altogether the assignment $M \mapsto v(f_*(l^\integers(M)))$
is lax symmetric monoidal and descends to a lax symmetric monoidal triangulated functor $$F \colon \D(r^\integers(P)) \to \D(E).$$

We claim $F$ is a symmetric monoidal full embedding with image the full localizing triangulated subcategory
spanned by the $E \otimes K^i$, $i \in \integers$.

We first equate the map $F$ on generating objects.
Let $A:=Qr^\integers(P)$. Note that $\D(A)$ is generated as
triangulated category with sums by the objects $A(i)$, $i \in \integers$.
Thus to show that $F$ is a full embedding it is sufficient
to show that $F$ induces isomorphisms on the Hom groups between
the $A(i)[k]$, $i,k \in \integers$. Now all functors involved
are compatible with the shifts $\_[k]$. The functor
$l^\integers$ is also compatible with the shifts $\_(i)$.
Since $\_(i)=\_ \otimes^\bL \unit(i)$
push forward along algebra maps is also compatible with
the shifts $\_(i)$. Thus $f_* \circ l^\integers$ is compatible
with the $\_(i)$. So we have $f_*(l^\integers(A(i)[k]))=P(i)[k]$.
Finally $(P(i)[k])_0 = E \otimes K^{-i} [k]$ and we get
$$F(A(i)[k])=E \otimes K^{-i} [k].$$

We set $T:=\_(i)[k]$. Let $\varphi \in \Hom_{\D(A)}(A,TA)$
By adjunction $\varphi$ corresponds to a map in
$\Hom_{\Ho(\caS^\integers)}(\unit,TA)$, this again corresponds to
a map $\psi \in \Hom_{\Ho(\caC^\integers)}(\unit,TP)$,
which is the same as a map in $\Hom_{\Ho \caC}(\unit,E \otimes K^{-i}[k])$.

Now the effect of $F$ on $\varphi$ is as follows.
First $\varphi$ is mapped to a map in
$$\Hom_{\D(l^\integers(A))}(l^\integers(A),T l^\integers(A)),$$
then by push forward to a map in $\Hom_{\D(P)}(P,TP)$.
By naturality and the properties of adjunctions
this map corresponds to $\psi$. We finally
see that both of the groups
$\Hom_{\D(A)}(A, TA)$ and $\Hom_{\D(E)}(E, E \otimes K^{-i}[k])$
are naturally isomorphic to $\Hom_{\Ho \caC}(\unit,E \otimes K^{-i}[k])$
and that the map induced by $F$ on these Homs corresponds to the
identity via these identifications.

We have proved that $F$ is a full embedding. Since $F$ is
compatible with sums we see that its image is closed under sums,
thus the statement about the image of $F$ follows.

We have to prove that $F$ is symmetric monoidal,
i.e. that the natural maps $F(X) \otimes F(Y) \to F(X \otimes Y)$
are isomorphisms.
We have to show that $v$ is symmetric monoidal (in a derived sense)
on the image of the functor $f_* \circ l^\integers$, where
on the left hand side we use the tensor product $\otimes_P$ and
on the right hand side we use $\otimes_E$.
Since the tensor product is triangulated and
compatible with sums it suffices for this
to show this property for the objects
$f_*(l^\integers(A(i)[k]))$, $i,k \in \integers$.

For this situation our claim follows from the definition 
of strong periodization and the
following general fact. If $B \in \Comm(\caC^\integers)$
(say cofibrant), $i(E) \to B$ a map of algebras, $X=B(i_1)[k_1]$,
$Y=B(i_2)[k_2]$, then the natural map
$$B_{-i_1} \otimes_E B_{-i_2} [k_1 + k_2] = v(X) \otimes_E v(Y)
\to v(X \otimes_B Y) = B_{-i_1 - i_2}[k_1 + k_2]$$
is given by the multiplication in $B$.
\end{proof}

\begin{proposition} \label{alg-periodizable}
 Let $E \in \Comm(\caC)$ be strongly periodizable and let $E \to E'$ be
a map in $\Comm(\caC)$. Then $E'$ is strongly periodizable.
\end{proposition}
\begin{proof}
 Let $Qi(E) \to i(E)$ be a cofibrant replacement and $Qi(E) \to P$ a strong periodization.
Then $i(E') \to i(E') \otimes_{Qi(E)}^\bL P$ is a strong periodization of $E'$.
\end{proof}

\section{Symmetric spectra and bimorphisms} \label{symm-bimor}

Our main references for symmetric spectra are \cite{hovey.symmspec} and
\cite{schwede.book}. We let $\caC$ be a left proper cellular symmetric monoidal model
category and $K \in \caC$ a cofibrant object. We let $\caS$ be the category of
symmetric $K$-spectra with the stable model structure as defined in \cite{hovey.symmspec}.
Its underlying category is the category of right $\mathrm{Sym}(K)$-modules in symmetric
sequences in $\caC$. Since $\mathrm{Sym}(K)$ is commutative it has a tensor product denoted
$\wedge$. We also denote by $K$ the image of $K$ in $\caS$.

Recall from \cite[Def. 8.9]{hovey.symmspec} the shift functor $s_-$ with the property $(s_- X)_n=X_{1+n}$.
Contrary to what is said in loc. cit. the $\Sigma_n$-action on $X_{1+n}$ is via the monomorphism
$\Sigma_n \to \Sigma_{1+n}$ which is induced by the strictly monotone embedding
$\{1,\ldots,n\} \to \{1, \ldots, n+1\}$ omitting $1$ in the target.

Recall from \cite[I.3.]{schwede.book} that a map $X \wedge Y \to Z$ in $\caS$ is a {\em bimorphism}
from $(X,Y)$ to $Z$, where a bimorphism consists of $\Sigma_p \times \Sigma_q$-equivariant maps
$$X_p \otimes Y_q \to Z_{p+q}$$ such that natural diagrams commute.

We denote by $\chi_{p,q} \in \Sigma_{p+q}$ the block permutation, see \cite{schwede.book}.

As noted in \cite[I.3.]{schwede.book} there is a natural morphism
$$(s_- X) \wedge Y \to s_-(X \wedge Y).$$

More generally if we are given a morphism $X \wedge Y \to Z$
with components $\alpha_{p,q} \colon X_p \otimes Y_q \to Z_{p+q}$ we exhibit a natural morphism
\begin{equation} \label{shift-prod}
(s_-^r X) \wedge (s_-^s Y) \to s_-^{r+s} Z
\end{equation}
having components $$\xymatrix{X_{r+p} \otimes Y_{s+q} \ar[r]^{\alpha_{r+p,s+q}} & Z_{r+p+s+q}
\ar[r]^{1 \times \chi_{p,s} \times 1} & Z_{r+s+p+q}}. $$

Suppose now that we are in the following situation: Let $M,N,P,M',N',P'$ be spectra and
let maps $M \to s_-^r M'$, $N \to s_-^s N'$, $P \to s_-^{r+s} P'$, $M \wedge N \to P$ and
$M' \wedge N' \to P'$ be given. We say that these maps are {\em compatible} if the diagram
$$\xymatrix{M_p \otimes N_q \ar[r] \ar[d] & P_{p+q} \ar[r] & P'_{r+s+p+q} \\
M'_{r+p} \otimes N'_{s+q} \ar[r] & P'_{r+p+s+q} \ar[ur]_{1 \times \chi_{p,s} \times 1} & }$$
commutes. This is the same as saying that the diagram
$$\xymatrix{M \wedge N \ar[r] \ar[d] & P \ar[d] \\
(s_-^r M') \wedge (s_-^s N') \ar[r] & s_-^{r+s} P'}$$
commutes, where the bottom horizontal map is the one from (\ref{shift-prod}).

We denote by $R$ and $Q$ the functorial fibrant and cofibrant replacement functors in $\caS$.

Note that by \cite[Theorem 8.10]{hovey.symmspec} we have natural isomorphisms $K^r \wedge (QX) \cong s_-^r (RX)$
in $\Ho \caS$. For later reference we note that for the proof of this fact a natural map
\begin{equation} \label{shift-trans}
 X \to (s_-X)^K
\end{equation}
is used in loc. cit. Contrary to what is said in loc. cit. this map involves a non-trivial
block permutation.

From the situation above we thus get maps $M \to K^r \wedge M'$, $N \to K^s \wedge N'$ and $P \to K^{r+s} \wedge P'$
in $\Ho \caS$, where we use the derived smash product. We form the diagram
\begin{equation} \label{shift-smash}
\xymatrix{K^r \wedge M' \wedge K^s \wedge N' \ar[d]^\cong & M \wedge N \ar[l] \ar[dd] \\
K^r \wedge K^s \wedge M' \wedge N' \ar[d] & \\
K^{r+s} \wedge P' & P \ar[l] }
\end{equation}
in $\Ho \caS$.

\begin{lemma} \label{mult-correct}
 Let the situation be as above. Suppose the given maps are compatible.
Suppose further that the maps $s_-^r M' \to s_-^r(RM')$, $s_-^s N' \to s_-^s(RN')$ and
$s_-^{r+s} P' \to s_-^{r+s} (RP')$ are stable equivalences.
Then the diagram (\ref{shift-smash}) commutes in $\Ho \caS$.
\end{lemma}

\begin{proof}
 We first note that for any symmetric spectrum $X$ there is a map of symmetric spectra
$$K \wedge X \to s_- X,$$ see \cite[Example I.2.18]{schwede.book}.
Iterating we get maps
\begin{equation} \label{tens-shift}
K^r \wedge X \to s_-^r X.
\end{equation}

The $n$-th component
is given by $$K^r \otimes X_n \cong X_n \otimes K^r \to X_{n+r} \overset{\chi_{n,r}}{\longrightarrow} X_{r+n},$$
where the twist map, the structure map of $X$ and the block permutation are used.

The identification $K^r \wedge X \cong s_-^r(RX)$ in $\Ho \caS$
is induced by the natural map $$K^r \wedge (QX) \to K^r \wedge X \to s_-^r X \to s_-^r(RX)$$
in $\caS$, since the transformations (\ref{tens-shift}), (\ref{shift-trans}) and the unit for the adjunction
$$K \wedge (\_) \leftrightarrow (\_)^K$$ are suitably compatible.

 We leave it to the reader to check that the square
\begin{equation} \label{tens-shift-square}
\xymatrix{K^r \wedge M' \wedge K^s \wedge N' \ar[r] \ar[d]^\cong & s_-^rM' \wedge s_-^sN'  \ar[dd] \\
K^r \wedge K^s \wedge M' \wedge N' \ar[d] & \\
K^{r+s} \wedge P' \ar[r] & s_-^{r+s} P'}
\end{equation}
commutes, where in the horizontal maps the maps (\ref{tens-shift}) are used and the right
vertical map is (\ref{shift-prod}).

We build the following diagram:

$$\tiny \xymatrix{& Q(K^r \wedge QM') \wedge Q(K^s \wedge QN') \ar[r] \ar[d]^\sim & Q(s_-^rM') \wedge Q(s_-^sN') \ar[d] & QM \wedge QN
\ar[d] \ar[l] \\
K^r \wedge K^s \wedge Q(QM' \wedge QN') \ar[r] \ar[d] & K^r \wedge QM' \wedge K^s \wedge QN' \ar[r] \ar[d] & s_-^rM' \wedge s_-^sN'  \ar[d] & M \wedge N \ar[d] \ar[l] \\
K^{r+s} \wedge QP' \ar[r] & K^{r+s} P' \ar[r] & s_-^{r+s} P' & P \ar[l] }.$$

The lower middle square commutes since (\ref{tens-shift-square}) commutes. All other
squares also commute. In the two top rows the left most horizontal maps are equivalences.
The composition of the left most maps in the last row also is an equivalence.
Thus viewing the diagram as a diagram in $\Ho \caS$ shows the claim.
\end{proof}

\section{Examples} \label{examples}

\subsection{Motivic cobordism}

In this section we are in the situation where
$l \colon \caS \to \caC$. The category $\caC$ will
be a model category modelling the stable motivic homotopy
category. In order that $\caC$ reveives a functor from
$\caS$ we have to use the following slight modification
of the usual versions for that category.
Let $S$ be a base scheme, Noetherian of finite
Krull dimension. We let $\mathrm{Sh}_S$
be the category of simplicial presheaves on 
$\Sm/S$, the category of smooth schemes over $S$,
endowed with a model structure that is Nisnevich and
$\bA^1$-local.
The category of symmetric $S^1_s$-spectra $\mathrm{Sp}_s(S)$,
$S^1_s$ the simplicial
circle, in $\mathrm{Sh}_S$ now receives a symmetric monoidal
left Quillen functor from $\caS$, and we let $\caC$ be
the category of symmetric $T$-spectra in $\mathrm{Sp}_s(S)$,
$T=\bA^1/(\bA^1 \setminus \{0\})$ the Tate object.
We leave it to the reader to verify that $\caC$ hits all of our
requirements.

To construct strong periodizations with $K$ the image of $T$ in $\caC$ we will nevertheless work
in the category of symmetric $T$-spectra in $\mathrm{Sh}_S$.
By transport of structure we will not loose anything.

Recall from \cite{PPR2} the strictly associative and commutative
model of the algebraic cobordism spectrum $\MGL$.
We will construct a strong periodization of it.

As in \cite{PPR2} we consider for any natural numbers $n,m$
the space $\bA^{nm} \cong \bA^m \times \cdots \times \bA^m$
($n$ factors) with the $\Sigma_n$-action coming from this
product decomposition. Instead of only considering $n$-planes
in this space we consider $k$-planes for all possible $k$.

We have to define a graded symmetric $T$-spectrum $\PMGL$.

We define the space $\PMGL_{r,n}$ in grade $r$ and spectrum level $n$:
if $r < -n$ we set $\PMGL_{r,n}= \mathrm{pt}$. Otherwise
set $\PMGL_{r,n}= \colim_m \Thom(\xi_{n+r,nm})$, where
$\xi_{n+r,nm}$ is the tautological vector bundle
over the Grassmannian $\Gr(n+r,nm)$ (the $\colim$ starts for such $m$ such that
$nm \ge n+r$). Having this definition the construction
works exactly as in \cite{PPR2}.

We have multiplication maps $\PMGL_{r_1,n_1} \wedge \PMGL_{r_2,n_2}
\to \PMGL_{r_1+r_2,n_1+n_2}$ which are
$\Sigma_{n_1} \times \Sigma_{n_2}$-equivariant, we have the units
$\mathrm{pt} \to \PMGL_{0,0}$ and $T \to \PMGL_{0,1}$.
This is the data we need to define a ring spectrum, see \cite[Def. I.1.3]{schwede.book}.
The structure maps of the individual spectra $\PMGL_r$ are
induced by the multiplication maps and the second unit.

\begin{theorem} \label{pmgl-per}
The graded spectrum $\PMGL$ is a strong periodization of
$\MGL$.
\end{theorem}

\begin{proof}
We first show that the individual spectra $\PMGL_r$ have the
correct homotopy type.

We note that the spectra $\PMGL_r$ are {\em semistable} in a motivically analogous
sense as in \cite[Th. 4.44]{schwede.book} by \cite[Proposition 3.2]{rso}. In particular the maps
$s_-^r \PMGL_k \to s_-^r(R \PMGL_k)$, $R$ a fibrant replacement functor
and $s_-$ the shift functor, see section \ref{symm-bimor}, are stable equivalences.

Let $r' \le r$ and $s = r - r'$. We define maps of spectra $\PMGL_r \to s_-^s \PMGL_{r'}$ as
follows.

There are maps $\Gr(n+r,nm) \to \Gr(n+r,(s+n)m)$ induced by the inclusion $\bA^{nm} \hookrightarrow \bA^{(s+n)m}$.
Those are covered by maps of the universal vector bundles inducing maps of Thom spaces.
Taking the colimit $m \to \infty$ we get maps $\PMGL_{r,n} \to \PMGL_{r',s+n}$ which are weak equivalences.
It is easily seen that these maps assemble to a map of spectra $\PMGL_r \to s_-^s \PMGL_{r'}$
which is a level equivalence. Thus by \cite[Theorem 8.10]{hovey.symmspec} we get an isomorphism
$\PMGL_r \cong K^s \wedge \PMGL_{r'}$ in $\Ho \caS$.
This shows that the $\PMGL_r$ have the correct homotopy type.

To show that the multiplication is the correct one we use lemma (\ref{mult-correct}) as follows:
let $m' \le m$, $n' \le n$, $p'=m' + n'$, $p = m+n$, $r=m-m'$, $ s= n-n'$,
$M=\PMGL_m$, $N=\PMGL_n$, $P=\PMGL_{m+n}$, $M'=\PMGL_{m'}$, $N'= \PMGL_{n'}$, $P'=\PMGL_{m' + n'}$,
$M \wedge N \to P$, $M' \wedge N' \to P'$ the multiplication maps, $M \to s_-^rM'$, $N \to s_-^sN'$, $P \to s_-^{r+s}P'$
the maps defined above. Then it is easily checked that these maps are compatible in the sense
of section \ref{symm-bimor}. Lemma (\ref{mult-correct}) now shows that the multiplication is the
correct one.
\end{proof}

Recall that the cellular objects $\D(\MGL)_\caT \subset \D(\MGL)$ comprise the full localizing
triangulated subcategory spanned by the Tate spheres $\MGL \wedge K^i$,
$i \in \integers$.

\begin{corollary} \label{mgl-rep}
There is a graded $E_\infty$-ring spectrum $A$ such
that $\D(A)$ is equivalent as tensor triangulated category
to $\D(\MGL)_\caT$.
\end{corollary}
\begin{proof}
 This follows from theorem (\ref{pmgl-per})
and theorem (\ref{rep-th}).
\end{proof}

For any motivic spectrum $X$ we denote the slices by $s_i(X)$, see \cite{voe-slice}. Since $\MGL$ is
effective (\cite{spitzweck-rel}) there is a map of ring spectra in the stable motivic homotopy category
$\MGL \to s_0 \MGL$. As noticed in \cite[Remark 7.2]{spitzweck-slice} this map can be realized as
a map of motivic $E_\infty$-ring spectra when our categorical assumptions hold. Moreover by
\cite{levine-htp} (and \cite{voevodsky-zero-slice} for fields of characteristic $0$)
and \cite[Cor. 3.3]{spitzweck-rel} this map is the map from $\MGL$ to the motivic Eilenberg MacLane spectrum $\MZ$ when
$S$ is the spectrum of a perfect field. 

Let $\D(\MZ)_\caT \subset \D(\MZ)$ be the full localizing
triangulated subcategory spanned by the Tate spheres $\MZ \wedge K^i$,
$i \in \integers$.

\begin{corollary} \label{mz-rep}
 Suppose $S$ is the spectrum of a perfect field and our categorical assumptions hold.
Then $\MZ$ is strongly periodizable. In particular there is a graded $E_\infty$-ring spectrum $A$ such
that $\D(A)$ is equivalent as tensor triangulated category
to $\D(\MZ)_\caT$.
\end{corollary}
\begin{proof}
 By the remarks above this follows from proposition (\ref{alg-periodizable}), theorem (\ref{pmgl-per})
and theorem (\ref{rep-th}).
\end{proof}

\begin{remark}
 Here we suppose our categorical assumptions hold. Then by the above discussion $s_0 \MGL$ is strongly periodizable,
in particular $(s_0\MGL)_\bQ$ is strongly periodizable. Thus if $S$ is regular \cite[cor. 6.4]{spitzweck-rel} states
that the Landweber theory $\LQ \cong (s_0\MGL)_\bQ$ has a strongly periodizable $E_\infty$-structure.
Since these spectra are rational it follows from our representation theorem
(\ref{rep-th}) that there is a graded rational cdga $A$ such that
$\D(A) \simeq \D(\LQ)_\caT$, the latter category being defined similarly as above.
We note that $\D(\LQ)_\caT$ is a good model for rational Tate motives
over any regular base.
\end{remark}

\begin{remark}
 There is a map of motivic $E_\infty$-ring spectra $\MGL \to \KGL$, see \cite[Prop. 5.10]{gepner-snaith}.
It thus follows from proposition (\ref{alg-periodizable}) and theorem (\ref{pmgl-per}) that
$\KGL$ is strongly periodizable.
\end{remark}

\begin{remark}
 Let $A$ be the graded topological spectrum of corollary (\ref{mgl-rep}) and let a complex
point of $S$ be given. Then topological realization provides us with a map of
graded $E_\infty$-ring spectra $\varphi \colon A \to \mathsf{PMU}$.
The topological realization functor $\D(A) \simeq \D(\MGL)_\caT \to \D(\MU)$ can be modelled by push forward along $\varphi$
and taking the zeroth component of the resulting graded $\mathsf{PMU}$-module.
\end{remark}

\subsection{Motivic cohomology}

In this section we will be in the situation where
$l \colon \caA \to \caC$. Let $k$ be a field. We will assume that $k$
is of characteristic $0$.
We first explain what $\caC$ is. We let $\shvnis(\smcor(k))$
be the category of Nisnevich sheaves with transfers on the
category of smooth schemes over $k$, see \cite{voevodsky.triangulated}. The category of complexes $\Cpx(\shvnis(\smcor(k)))$
has an $\bA^1$- and Nisnevich local symmetric monoidal model structure such
that the canonical functor from $\caA$ is symmetric monoidal left Quillen and
such that $\bbT:=S^0 \ztr(\PP^1,\{\infty\})$ is cofibrant ($S^0 X$ denotes the complex where $X$ sits in
degree $0$).
The category $\caC$ is defined to be the category of symmetric $\bbT$-spectra in $\Cpx(\shvnis(\smcor(k)))$
with the stable model structure defined in \cite{hovey.symmspec}.
The object $K$ is defined to be the image of $\bbT$ in $\caC$.

\begin{theorem} \label{mot-per}
Suppose our cartegorical assumptions hold.
Then the unit sphere in $\caC$ has a strong periodization.
\end{theorem}
\begin{proof}
 For any $X,U \in \Sm/k$ we let $\zequi(X,r)(U)$ the free abelian group generated by closed integral
subschemes of $X \times_k U$ which are equidimensional of relative dimension $r$ over $U$,
see \cite{friedlander-voevodsky.bivariant}. The assignment $U \mapsto \zequi(X,r)(U)$ has
the structure of a Nisnevich sheaf with transfers on $\Sm/k$. Note that we have natural bilinear
maps $\zequi(X,r)(U) \times \zequi(Y,s)(U) \to \zequi(X \times_k Y, r+s)(U)$ which are functorial
for finite correspondences. Thus we get maps $$\zequi(X,r) \otimes \zequi(Y,s) \to \zequi(X \times_k Y,r+s),$$
where the tensor product in $\shvnis(\smcor(k))$ is used.

We are going to define a ring object $\sP$ in $\caC^{\integers_{\le 0}}$. For non-negative
integers $r$ we let the spectra $\sP_{-r}$ be given by $\sP_{-r,n}=S^0 \zequi(\bA^n,r)$ with the obvious action
of $\Sigma_n$ and with
multiplication maps $$\sP_{-r_1,n_1} \otimes \sP_{-r_2,n_2} \to \sP_{-r_1-r_2,n_1+n_2}$$ given
by the above multiplication of cycles. These are $\Sigma_{n_1} \times \Sigma_{n_2}$-equivariant.
The two units $\integers \to \sP_{0,0}$ and $\bbT=S^0(\zequi(\PP^1,0)/\zequi(\{\infty\},0)) \to \sP_{0,1}=S^0 \zequi(\bA^1,0)$
are the natural ones. One checks easily that we get a commutative monoid in $\caC^{\integers_{\le 0}}$.
Moreover the unit map is an equivalence in degree $0$ by \cite[Prop. 4.1.5]{voevodsky.triangulated}.

We claim that $\sP$ is a semi periodization of the unit.
We first show that the individual spectra $\sP_{-r}$ have the correct homotopy type.

We claim $s_-^r \sP_{-r} \simeq \sP_0$,
where $s_-$ is the shift functor, see section \ref{symm-bimor}. Indeed, flat pullback of cycles along the projections
$\bA_k^{r+n} \cong \bA_k^r \times_k \bA_k^n \to \bA_k^n$
gives maps
$$j_n \colon \zequi(\bA^n,0) \to \zequi(\bA^{r+n},r).$$
We claim that these maps assemble to a map of spectra $\sP_0 \to s_-^r \sP_{-r}$.
We have to show that the $j_n$ are compatible with the structure maps
$$S^0 \zequi(\bA^n,0) \otimes \bbT \to S^0 \zequi(\bA^{n+1},0)$$ and
$$S^0 \zequi(\bA^{r+n},r) \otimes \bbT \to S^0 \zequi(\bA^{r+n+1},r).$$
This follows since the structure maps are given by mutliplication of cycles from the right
and we use flat pullback on the left.

We claim the map $j \colon \sP_0 \to s_-^r \sP_{-r}$ is a level equivalence.
Since $\sP_0$ is an $\Omega$-spectrum it follows then that $\sP_{-r}$ is also an $\Omega$-spectrum
and $s_-^r \sP_{-r} \simeq (Rs_-)^r \sP_{-r}$. It follows from \cite[Theorem 8.10]{hovey.symmspec}
that $\sP_{-r} \wedge^\bL K^r \simeq \sP_0$. Thus $\sP_{-r}$ will have the correct homotopy type.

We prove that the $S^0 j_n$ are equivalences. First note that by
\cite[Prop. 5.7 2.]{friedlander-voevodsky.bivariant} the presheaves
$\zequi(X,r)$ are pretheories in the sense of 
\cite[sec. 5]{friedlander-voevodsky.bivariant}.
For any presheaf $F$ on $\Sm/k$ with values in abelian groups denote by $\underline{C}_* F$
the complex associated to the simplicial presheaf
$U \mapsto F(\Delta^\bullet \times U)$.

The proof of \cite[Prop. 5.5 1.]{friedlander-voevodsky.bivariant}
shows that for any pretheory $F$ and $U \in \Sm/k$ we have isomorphisms
\begin{equation} \label{Nis-cdh}
\bH^i_{\mathit{Nis}}(U, (\underline{C}_* F)_{\mathit{Nis}})
\cong \bH^i_{\mathit{cdh}}(U, (\underline{C}_* F)_{\mathit{cdh}}).
\end{equation}
Now to show that the $S^0j_n$ are equivalences it is sufficient
to show that the $\underline{C}_* j_n$ are Nisnevich-local equivalences.

This follows from (\ref{Nis-cdh}),
\cite[Prop. 8.3 1.]{friedlander-voevodsky.bivariant}
and the definition of the bivariant cycle cohomology
\cite[Def. 4.3]{friedlander-voevodsky.bivariant}.

To show that $\sP$ is a semi periodization we are left to show that
the multiplication is the correct one. We apply lemma (\ref{mult-correct}) with
$M=N=P=\sP_0$, $M'=\sP_{-r}$, $N'=\sP_{-s}$, $P'=\sP_{-r-s}$. The maps
$M \to s_-^r M'$, $N \to s_-^s N'$, $P \to s_-^{r+s} P'$ are the maps $j$ constructed above.
The maps $M \wedge N \to P$ and $M' \wedge N' \to P'$ are the multiplication maps.
By inspection these maps are compatible in the sense of paragraph \ref{symm-bimor}.
Moreover the maps $s_-^r M' \to s_-^r(RM')$, $s_-^s N' \to s_-^s(RN')$ and
$s_-^{r+s} P' \to s_-^{r+s} (RP')$ are stable equivalences since all appearing spectra are $\Omega$-spectra.
Now lemma (\ref{mult-correct}) indeed says that the multiplication is the correct one.

Having constructed a semi periodization it follows from proposition (\ref{semi-per}) that under
the categorical assumptions the unit in $\caC$ has a periodization.
\end{proof}

We let $\DM(k):= \Ho \caC$ and $\DM(k)_\caT$ be the full localizing triangulated
subcategory of $\DM(k)$ generated by the $\integers(i)$ where $\integers(i)=K^i$.

\begin{corollary} \label{mot-rep}
 Suppose our categorical assumptions hold. Then there is an $E_\infty$-algebra
$A$ in $\caA^\integers$ such that $\D(A)$ is equivalent as tensor triangulated category to $\DM(k)_\caT$.
\end{corollary}
\begin{proof}
 This follows from theorem (\ref{mot-per})
and theorem (\ref{rep-th}).
\end{proof}

Let now $X$ be a separated Noetherian scheme of finite Krull dimension over $k$.
Then $\check{\mathrm{C}}$isinski-Deglise have constructed a model category $\caC_X$ built up from finite correspondences
for smooth schemes over $X$ such that $\Ho(\caC_X)$ is a good model for $\DM(X)$. The model category
$\caC_X$ receives a symmetric monoidal left Quillen functor form $\caC$. Thus it follows from
theorem (\ref{mot-per}) that the tensor unit in $\caC_X$ is strongly periodizable.
Let $\DM(X)_\caT$ be the full localizing triangulated
subcategory of $\DM(X)$ generated by the $\integers(i)$. Then we get

\begin{corollary}
 Suppose our categorical assumptions hold. Then there is an $E_\infty$-algebra
$A$ in $\caA^\integers$ such that $\D(A)$ is equivalent as tensor triangulated category to $\DM(X)_\caT$.
\end{corollary}
\noindent

\bibliographystyle{plain}
\bibliography{per}

\begin{center}
Fakult{\"a}t f{\"u}r Mathematik, Universit{\"a}t Regensburg, Germany.\\
e-mail: Markus.Spitzweck@mathematik.uni-regensburg.de
\end{center}

\end{document}